\documentclass{amsart}
\usepackage{mathrsfs,amssymb}
\newtheorem{theorem}{Theorem}[section]
\newtheorem{lemma}[theorem]{Lemma}
\newtheorem{proposition}[theorem]{Proposition}
\newtheorem{corollary}[theorem]{Corollary}
\newtheorem{definition}[theorem]{Definition}

\theoremstyle{definition}
\newtheorem{remark}[theorem]{Remark}
\newtheorem{example}[theorem]{Example}
\numberwithin{equation}{section}
\newcommand{\CC}{\mathbb{C}}

\newcommand{\ZZ}{\mathbb{Z}}

\renewcommand{\[}{\begin{equation}}
\renewcommand{\]}{\end{equation}}

\title[RKHS generated by the binomial coefficients]{Reproducing kernel Hilbert spaces
 generated by the binomial coefficients}

\begin{document}
\author[D. Alpay]{Daniel Alpay}
\address{(DA) Department of Mathematics \newline
Ben Gurion University of the Negev \newline P.O.B. 653, \newline
Be'er Sheva 84105, \newline ISRAEL} \email{dany@math.bgu.ac.il}
\author[P. Jorgensen]{Palle Jorgensen}
\address{(PJ)
Department of Mathematics\newline 14 MLH \newline The University
of Iowa, Iowa City,\newline IA 52242-1419 USA}
\email{jorgen@math.uiowa.edu}

\thanks{D. Alpay thanks the
Earl Katz family for endowing the chair which supported his
research. The research of the authors was supported in part by the
Binational Science Foundation grant 2010117.}

 \maketitle

\begin{abstract}
We study a reproducing kernel Hilbert space of functions defined on
the positive integers and associated to the binomial coefficients. We introduce
two transforms, which allow us to develop a related harmonic
analysis in this Hilbert space. Finally, we mention connections with the
theory of discrete analytic functions, statistics, and with the quantum case.
\end{abstract}
\tableofcontents
\section{Introduction}
\label{sec1}
We show that the binomial coefficients form an orthonormal basis
(ONB) in a naturally formed reproducing kernel Hilbert space
$\mathcal H(K)$. We find explicit properties and formulas for the
kernel, and we identify two pairs of transforms, again with
explicit formulas. With the binomial functions forming an ONB in
$\mathcal H(K)$, one naturally then gets transforms as isometric
isomorphisms between $\mathcal H(K)$ and $\ell^2(\mathbb Z_+)$. We
write down two such transforms, and for each we compute its
inverse. Our work is motivated by applications to discrete
analytic functions \cite{2012arXiv1208.3699A}. This is followed up here, and we
outline additional applications to combinatorial probability
theory. In \cite{2012arXiv1208.3699A}, we introduce a product (in the sense of
hypergroups \cite{MR2606478}) in spaces of discrete analytic functions.
We use our present transforms for computing this product.
In the last section, as  a corollary to  Theorem 7.3 we establish a significant statistical property for our one-parameter family of kernels  $K_{\lambda}$, as well as for the corresponding spectral transforms: we show for each kernel within our family of kernels  $K_{\lambda}$  there is a probability space; hence an associated indexed family of probability spaces $( \Omega,\mathcal F,P^\lambda )$ . We prove that as the value of the parameter $\lambda$  vary, the corresponding measures  $P^\lambda$ are mutually singular.
For readers not familiar with reproducing kernels, we suggest
\cite{MR2002b:47144,aron,meschkowski,saitoh,schwartz}.
For a sample of references to combinatorial probability, see e.g,
\cite{MR2925971,MR2668912,MR2784561}.\\

Motivated by our previous work on discrete analytic functions
\cite{2012arXiv1208.3699A} we study
a reproducing kernel space of functions from
$\mathbb Z_+=\left\{0,1,2,\ldots\right\}$ into $\mathbb R$ associated to the binomial
coefficients.\\

We study the reproducing kernel Hilbert space (RKHS) $\mathcal H(K)$  generated by the kernel $K(x, y)$ 
consisting of the numbers in the Pascal triangle, i.e., for every  $(x, y)$   in 
$\mathbb Z_+ \times\mathbb Z_+$, set $K(x, y) = \begin{pmatrix}x+y\\x\end{pmatrix}$.
Using the Chu-Vandermonde formula, one sees that $K(x,y)$ is positive definite in the sense of kernel theory,  
hence the RKHS. We study a number of analytic features of $\mathcal H(K)$:  We find families of orthonormal bases in 
$\mathcal H(K)$, and we prove associated transform theorems. Moreover we compute these transforms explicitly, making use of 
infinite square matrices of upper and lower triangular form.  We further show that our transforms facilitate analysis on 
$\mathcal H(K)$. We further sketch a one-parameter deformation family with application .\smallskip

In the remaining part of this section we give results characterizing vectors in their respective reproducing kernel Hilbert spaces; see
Theorem \ref{tm11}.\\

{\bf Definitions.} Let for $x,n\in\mathbb Z_+$ such that $n\le x$
\[
\begin{pmatrix}x\\ n\end{pmatrix}=\frac{x(x-1)\cdots (x-n+1)}{n!}
\]
the binomial coefficient. We set
\[
\label{eqen} e_n(x)=\begin{cases}\,\,0\,,
\quad\hspace{2.9cm}{\rm if}\quad x<n,\\
&\\
\begin{pmatrix}x\\ n\end{pmatrix}=\frac{x(x-1)\cdots (x-n+1)}{n!},\,{\rm if}
\quad x\ge n,
\end{cases}
\]
and
\[
\label{rksum} K(x,y)=\sum_{n=0}^\infty e_n(x)e_n(y)=\begin{pmatrix}x+y\\x\end{pmatrix} \quad
x,y\in\mathbb Z_+.
\]
The sum \eqref{rksum} is always finite and well defined, and
defines a positive definite function on $\mathbb Z_+$.\\

Using the Chu-Vandermonde formula (see for instance \cite[formula (10), p. 217]{tucker})
\[
\begin{pmatrix}x+y\\ t\end{pmatrix}=\sum_{k=0}^t\begin{pmatrix}x\\ k\end{pmatrix}\begin{pmatrix}y\\ t-k\end{pmatrix}
\]
(with $t,x,y\in\mathbb N_0$ and $t\le y$)
 one can rewrite $K(x,y)$ as
\[
\label{rksum2}
K(x,y)=\sum_{n=0}^{x\wedge y}\begin{pmatrix}x\\
n\end{pmatrix}
\begin{pmatrix}y\\ n\end{pmatrix},
\]
or, equivalently,
\[
\label{rksum3}
K(x,y)=\sum_{n=0}^{x\wedge y}\frac{x^{[n]}y^{[n]}}{(n!)^2},
\]
where we have set
\[
\label{xnstr}
x^{[n]}=x(x-1)\cdots (x-n+1).
\]
Writing the kernel $K(x,y)$ as an infinite matrix $K$ we note that we have
\[
\label{eq:LLstar}
K=LL^*,
\]
where $L$ denote the following infinite lower triangular matrix
consisting of the binomial coefficients
\[
\label{eq:l}
\begin{split}
L&=\\
&\hspace{-5mm}=
\small{
\left(
\begin{array}{cccccccccccc}
1&0&0&0&0&0&0&\cdots&\quad&\cdot&0&\cdots\\
1&1&0&0&0&0&0&\cdots&\quad&\cdot&0&\cdots\\
1&2&1&0&0&0&0&\cdots&\quad&\cdot&0&\cdots\\
1&3&3&1&0&0&0&\cdots&\quad&\cdot&0&\cdots\\
1&4&6&4&1&0&0&\cdots&\quad&\cdot&0&\cdots\\
\vdots&\vdots& & & & & & & & & &\\
1&n&\begin{pmatrix}n\\2\end{pmatrix}&\cdots&\begin{pmatrix}n\\2\end{pmatrix}&n&1&0
&0&\cdots&0&\cdots \\
1&n+1&\begin{pmatrix}n+1\\2\end{pmatrix}&\cdots&&\begin{pmatrix}n+1\\2\end{pmatrix}
&n+1&1&0
&\cdots&0&\cdots \\
\vdots&\vdots&\vdots&&&\vdots&\vdots&\vdots&1&&&\\
\vdots&\vdots&\vdots&&&\vdots&\vdots&\vdots&\vdots&&&
\end{array}
\right)}
\end{split}
\]

We denote by $\mathcal H(K)$ the reproducing kernel of real
valued functions with domain $\mathbb Z_+$ and with reproducing
kernel $K(x,y)$.
This latter space consists of all functions
of the form
\[
g(x)=\sum_{n=0}^\infty a_ne_n(x),
\]
where the $a_n\in\mathbb R$, and with norm
\[
\|g\|^2_{\mathcal H(K)}=\sum_{n=0}^\infty a_n^2.
\]
See \cite[corollary 4, p. 169]{schwartz}.\\

It follows from \eqref{eq:LLstar} that we have:

\begin{theorem}
\label{tm11}
Let $f\in\mathbb R^{\mathbb Z_+}$. Then,
\[
\sum_{x\in\mathbb Z_+}K(\cdot, x)f(x)\in\mathcal H(K)\quad\iff\quad L^*f\in
\ell^2(\mathbb Z_+).
\]
\end{theorem}

{\bf Organization.}
The paper consists of seven sections and of an appendix
besides the introduction,
and its outline is as follows:  In Section \ref{sec3} we review
some results on binomial
coefficients. Sections \ref{sec5}-\ref{sec7} contain the main results of the paper,
on the above mentioned transforms. The first transform, which we call the {\sl
binomial Fourier transform}, and its inverse, are studied in
Section \ref{sec5}. A natural isomorphism between $\mathcal H(K)$
and the Hardy space $\mathbf H^2(\mathbb D)$ is defined in that
section. The following two sections, Section
\ref{sec7} is devoted to the second transform. In Section
\ref{sec8} we study some links with the theory of discrete
analytic functions. In Section \ref{sec9} we considered the case
of $q$-binomial coefficients. We present some applications in the last section.

\section{Some formulas for binomial coefficients}
\label{sec3}
\setcounter{equation}{0}
In this section we establish some duality relations for the
binomial functions in \eqref{eqen}, beginning with Lemma 3.1.
This lemma in turn will be used in our results from sections
\ref{sec5} through \ref{sec7} dealing with the two transforms. The
transforms throw light on the RKHS of the binomial functions, but
they also imply new formulas for these functions,
%generating
%matrix-functions (Theorem \ref{tm32}),
linear relations (Theorem
3.3), and their use in the study of discrete analytic functions
(section \ref{sec8}).\\

We make the following summation convention. In computations below,
we will be using summations over index-values in finite or
infinite subsets of $\mathbb Z_+ = \left\{0, 1, 2,
\cdots\right\}$; in some cases, multiple summations inside a
single computation. These summations will then always entail
ranges over summation indices $x, y, k, n, \cdots$ which are
limited by choice of segments in $\mathbb Z_+$; and the respective
summations will be specified by the segment-endpoints.

\begin{lemma}
\label{le:eq}
Let $m,n\in\mathbb Z_+$ be such that $m\le n$. Then,
\[
\sum_{j=m}^n(-1)^{m+j}\begin{pmatrix}n
\\j\end{pmatrix}\begin{pmatrix}j\\ m\end{pmatrix}=\delta_{m,n}
\]
\end{lemma}
\begin{proof} The result is clear when $m=n$. Let us assume now $m<n$. We first note that
\[
\begin{split}
\begin{pmatrix}n\\ j\end{pmatrix}\begin{pmatrix}j\\ m\end{pmatrix}&=\frac{n!}{(n-j)!j!}\frac{j!}{(j-m)!m!}\\
&=\frac{n!}{(n-j)!(j-m)!m!}\\
&=\frac{n!}{(n-m)!m!}\frac{(n-m)!}{(n-j)!(j-m)!}\\
&=
\begin{pmatrix}n\\ m\end{pmatrix}\begin{pmatrix}n-m\\j- m\end{pmatrix}.
\end{split}
\]
Hence,
\[
\begin{split}
\sum_{j=m}^n(-1)^{m+j}\begin{pmatrix}n
\\j\end{pmatrix}\begin{pmatrix}j\\ m\end{pmatrix}&=\sum_{j=m}^n(-1)^{m+j}
\begin{pmatrix}n
\\m\end{pmatrix}
\begin{pmatrix}n-m\\j- m\end{pmatrix}\\
&=\begin{pmatrix}n\\m\end{pmatrix}\sum_{j=m}^n(-1)^{m+j}\begin{pmatrix}n-m\\j- m\end{pmatrix}\\
\intertext{\mbox{\rm and setting $\ell=j-m$},}
&=\begin{pmatrix}n\\m\end{pmatrix}\sum_{\ell=0}^{n-m}(-1)^\ell\begin{pmatrix}n-m\\ \ell \end{pmatrix}\\
&=\begin{pmatrix}n\\m\end{pmatrix}(1-1)^{n-m}\\
&=0.
\end{split}
\]
\end{proof}

The following result is a corollary of Lemma \ref{le:eq},
and plays an important role in the study of the
$\mbox{}^\blacktriangle$-transform. See Theorem \ref{tm:inverse1}
below.

\begin{theorem}
\label{tm33}
Let $K(\ell,m)$ be defined by \eqref{rksum2}. Then it holds that:
\begin{equation}
\label{eq:inverse11}
\sum_{\ell=0}^n(-1)^{n+\ell}\begin{pmatrix}n\\
\ell\end{pmatrix}K(\ell,m)=
\begin{cases}
\begin{pmatrix}m\\n\end{pmatrix},\quad if\quad
m\ge n,\\
\\
\,\,\,0\,,\,\,\, \quad \hspace{3mm}if\quad m<n.
\end{cases}
\end{equation}
\end{theorem}

\begin{proof} By the definition of $K$, we have
\[
\nonumber
\begin{split}
\sum_{\ell=0}^n (-1)^{n+\ell}\begin{pmatrix}n \\
\ell\end{pmatrix}K(\ell,m)&=
\sum_{\ell=0}^n\sum_{j=0}^{m\wedge\ell} (-1)^{n+\ell}\begin{pmatrix}n \\
\ell\end{pmatrix}\begin{pmatrix}m \\
j\end{pmatrix}\begin{pmatrix}\ell \\
j\end{pmatrix}\\
&=\sum_{j=0}^{n\wedge m}\begin{pmatrix}m \\
j\end{pmatrix}\left(\sum_{\ell=j}^n
(-1)^{n+\ell}\begin{pmatrix}n \\
\ell\end{pmatrix}\begin{pmatrix}\ell \\
j\end{pmatrix}\right)\\
\intertext{\mbox{\rm and, applying Lemma \ref{le:eq}}}
&=\sum_{j=0}^{n\wedge m}\begin{pmatrix}m \\
j\end{pmatrix}\delta_{j,n}\\
&=\begin{cases}
\begin{pmatrix}m\\n\end{pmatrix},\quad if\quad
m\ge n,\\
\\
\,\,\,0\,,\,\,\, \quad \hspace{3mm}if\quad m<n.
\end{cases}
\end{split}
\]
\end{proof}

\begin{lemma}
\label{newlemma}
Let $K(x,y)$ be defined by \eqref{rksum2}. Then it holds that:
\begin{eqnarray}
\label{formula1}
K(x,x)&=&\sum_{n=0}^x
\begin{pmatrix}x\\ n\end{pmatrix}^2=\begin{pmatrix}2x\\ x\end{pmatrix}=e_x(2x),\\
|K(x,y)|&\le& K(x,x)^{1/2}K(y,y)^{1/2}=\left(\begin{pmatrix}2x\\
x\end{pmatrix}
\begin{pmatrix}2y\\ y\end{pmatrix}\right)^{1/2}.
\end{eqnarray}
\end{lemma}

\begin{proof} 
The first claim follows from setting $x=y$ in \eqref{rksum2}.
The second claim follows from the Cauchy-Schwarz inequality, since $K$ is positive
definite.
\end{proof}

\section{The binomial Fourier transform}
\label{sec5}
\setcounter{equation}{0}
In this section we introduce the first family of transforms.
We then use them in order to characterize functions on
$\mathbb Z_+$ belonging to $\mathcal H(K)$,
their properties, both analytic and algebraic. For example in Theorem \ref{tm54},
we give a natural isomorphism between $\mathcal H(K)$ and the Hardy space of the
disk. In Theorem \ref{tm512} we show that the orthonornal vectors $e_n$  in
$\mathcal H(K)$ from \eqref{eqen}  generate an algebra, or rather a hypergroup.
See \cite{MR2606478}.\\

The definition of the $\mbox{}^\blacktriangle$-transform does not
need the hypothesis that the sequence is in $\ell_2(\mathbb
Z_+)$, and is as follows:

\begin{definition}
\label{def51}
Let $f$ be a function from $\mathbb Z_+$ into
$\mathbb R$. The binomial Fourier transform of $f$ is defined by
\begin{equation}
\label{direct}
f^\blacktriangle(x)=\sum_{j=0}^x \begin{pmatrix}x\\
j\end{pmatrix}f(j),\quad x=0,1,2\ldots.
\end{equation}
\end{definition}
Thus,
\begin{equation}
\begin{split}
f^\blacktriangle(x)&=f(0)+xf(1)+\frac{x(x-1)}{2}f(2)+\frac{x(x-1)(x-2)}{6}f(3)+\cdots\\
&=\sum_{n=0}^\infty f(n)e_n(x).
\end{split}
\label{newsum}
\end{equation}

The sums in \eqref{direct} are finite, and $f^\blacktriangle$
exists for every function from $\mathbb Z_+$ into $\mathbb R$. We note that
\eqref{direct} can be written as
\[
\label{direct21}
f^{\blacktriangle}=Lf=L(1)f,
\]
where $L(\lambda)$ is defined in Section \ref{lastsec} below by \eqref{eq1}.\\

The image of the function $f(j)\equiv 0$ is the zero
function. The image of $f(j)=(-1)^j$ is the function
\[
f^\blacktriangle (x)=\delta_{0,x}.
\]
while the image of $f(j)=1$ is $f^\blacktriangle (x)=2^x$.\\

Another example of interest is:
\begin{example}
Let $f_a(j)=a^{*j}$, where $a\in \mathbb C$. Then,
\begin{equation}
\label{newsum1}
f^\blacktriangle(x)=\sum_{n=0}^\infty  a^{*n}e_n(x)
\end{equation}
is given by
\begin{equation}
\label{newsum0}
f_a^\blacktriangle(x)
=\sum_{j=0}^x\begin{pmatrix}x\\ j\end{pmatrix}a^{*j}=(1+a^*)^x,\quad x=0,1,2,\ldots
\end{equation}
\end{example}

Let now $\mathbf H^2(\mathbb D)$ denote
the Hardy space of complex-valued functions $f(z)=\sum_{n=0}^\infty a_nz^n$
analytic in the open unit disk and such that
\[
\|f\|_2^2:=\sum_{n=0}^\infty|a_n|^2<\infty.
\]
We define a map $\psi$ from the complexified of $\mathcal H(K)$ into $\mathbf H^2(\mathbb D)$ by
\begin{equation}
\psi(e_n)=z^n.
\end{equation}

\begin{theorem}
%\mbox{}\\
The map $\psi$ is unitary from the complexified of $\mathcal H(K)$ onto the Hardy space
$\mathbf H^2(\mathbb D)$. Moreover it associates to the function $f_a:$
$x\mapsto (1+a^*)^x,\, x\in\mathbb Z_+$ the
function $z\mapsto \frac{1}{1-za^*},\,z\in\mathbb D$. In particular,
for $a,b$ in the open unit disk $\mathbb D$
\begin{equation}
\langle  f_a^\blacktriangle,f_b^\blacktriangle\rangle_{\mathcal H(K)}=\frac{1}{1-a^*b}.
\label{newcauchy}
\end{equation}
\label{tm54}
\end{theorem}
\begin{proof}
The map is unitary since it maps an orthonormal basis onto an orthonormal basis. The claim on the image of the sequence
$f_a$ follows from \eqref{newsum0} and \eqref{newsum1}. Finally, \eqref{newcauchy} follows from the unitarity of $\psi$, the formula
\[
\frac{1}{1-za^*}=\sum_{n=0}^\infty z^na^{*n}
\]
and the definition of the inner product of $\mathbf H^2(\mathbb D)$.
\end{proof}

The next proposition is used to prove that the map $\mbox{}^\blacktriangle$ is
unitary from
$\ell^2(\mathbb Z_+)$ onto $\mathcal H(K)$. See Theorem \ref{tm:inverse1}.

\begin{proposition}
\label{eigenvalue}
Let $e_n$ defined by \eqref{eqen}, and let
\[
\label{eqrfn}
\widetilde{e_n}(m)=(-1)^{n+m}e_n(m), \quad and\quad \delta_n(m)=\delta_{n,m}.
\]
Then
\[
\label{eqen1}
\delta_n^\blacktriangle=e_n\quad and\quad
(\widetilde{e_n})^\blacktriangle=\delta_{n}.
\]
\end{proposition}
\begin{proof}
The first claim follows from the definition of $e_n$. As for the second claim, we have:
\[
\nonumber
\begin{split}
(\widetilde{e_n})^\blacktriangle(m)&=\sum_{j=0}^m\begin{pmatrix}m\\j\end{pmatrix}(-1)^{n+j}e_n(j)\\
&=\begin{cases}\,\, 0,\quad {\rm if}\quad m<n \quad(\mbox{\rm since $e_n(j)=0$ for $j<n$})\\
\\
\sum_{j=n}^m\begin{pmatrix}m\\j\end{pmatrix}\begin{pmatrix}j\\n\end{pmatrix}(-1)^{n+j},\quad {\rm if}\quad m\ge n.
\end{cases}
\end{split}
\]
We conclude by using Lemma \ref{le:eq} to compute this last sum.
\end{proof}

It follows from \eqref{newcauchy} that most, if not all,
the problems considered in $\mathbf H^2(\mathbb D)$
can be transposed in a natural way in the space $\mathcal H(K)$.\smallskip

We now define the function $\epsilon_n(j)=e_j(n)$. Thus $(\epsilon_n(j))$ can be identified with the finitely nonzero
sequence of $\ell^2(\mathbb Z_+)$ whose nonzero terms are the first $n+1$ terms given by $\begin{pmatrix}n \\j\end{pmatrix},\, j=0,
\ldots, n$.

\begin{lemma}
\label{la:direct}
Let $n,m\in \mathbb Z_+$. Then,
\begin{eqnarray}
\label{direct1}
(\epsilon_n^\blacktriangle)(x)&=&K(x,n)\\
\label{direct2}
\langle
\epsilon_n,\epsilon_m\rangle_{\ell^2(\mathbb Z_+)}&=&
 \langle K(\cdot, n),K(\cdot, m)\rangle_{\mathcal
H(K)}.
\end{eqnarray}
\end{lemma}

\begin{proof}
By definition of the transform,
\[
(\epsilon_n^\blacktriangle)(x)=\sum_{j=0}^x \begin{pmatrix}x\\
j\end{pmatrix}\epsilon_n(j)= \sum_{j=0}^{x\wedge n}
\begin{pmatrix}x\\
j\end{pmatrix}\begin{pmatrix}n
\\j\end{pmatrix}=K(x,n),
\]
and so \eqref{direct1} holds. To prove \eqref{direct2}, it
suffices to notice that
\[
\begin{split}
\langle \epsilon_n,\epsilon_m\rangle_{\ell^2(\mathbb
Z_+)}=\sum_{j=0}^{n\wedge m}\begin{pmatrix}n\\
j\end{pmatrix}\begin{pmatrix}m\\
j\end{pmatrix}
=K(n,m).
\end{split}
\]
\end{proof}

\begin{theorem}
\label{tm:inverse1}
The map $f\mapsto f^\blacktriangle$ defines a
unitary mapping from $\ell_2(\mathbb Z_+)$ onto $\mathcal H(K)$,
and its inverse is given by either formulas
\[
g^\blacktriangledown=\sum_{n=0}^\infty a_n(-1)^n\delta_n\quad for
\quad g=\sum_{n=0}^\infty a_ne_n\in\mathcal H(K),
\]
or
\[
\label{tr:inverse}
g^\blacktriangledown(n)=\sum_{\ell=0}^n \begin{pmatrix}n
\\  \ell\end{pmatrix}g(\ell)(-1)^{\ell-n}.
\]
\end{theorem}

We note that \eqref{tr:inverse} can be rewritten as
\[
g^\blacktriangledown=L(-1)g,
\]
where $L(\lambda)$ is defined by \eqref{eq1} below.\\

\begin{proof}[Proof of Theorem \ref{tm:inverse1}.]
We present two proofs. The first is based on Proposition \ref{eigenvalue}. Let
\[
f=\sum_{n=0}^Na_n\delta_n
\]
be a real valued function on $\mathbb Z_+$ with finite support. Then,
\[
f^\blacktriangle=\sum_{n=0}^Na_ne_n\,\,\in\mathcal H(K),
\]
and
\[
\|f^\blacktriangle\|^2_{\mathcal H(K)}=\|f\|^2_{\ell^2(\mathbb
Z_+)}=\sum_{n=0}^Na_n^2.
\]
The result follows by taking limits.\\

The second proof uses Lemma \ref{la:direct}.
The functions $\epsilon_0,\epsilon_1,\ldots$ span a dense set of
$\ell_2(\mathbb Z_+)$, and therefore \eqref{direct2} extends to
an isometry. The isometry is onto since the kernels span a dense
set of $\mathcal H(K)$. Therefore the transform
$f^\blacktriangle$ is unitary. To prove its inverse is given by
\eqref{tr:inverse} we note that \eqref{eq:inverse11} can be rewritten as
\[
(K^\blacktriangledown_m)(n)=\epsilon_m(n),\quad n\in\mathbb Z_+.
\]
\end{proof}

As a corollary of the above, the next results offer an answer to
the following question: {\sl Given a function $f$ on $\mathbb
Z_+$, when does $f$ belong to the reproducing kernel Hilbert space
$\mathcal H(K)$ with reproducing kernel \eqref{rksum2}?}

\begin{corollary}
\label{cor57}
Let $f\,\,:\,\,\mathbb Z_+\,\,\longrightarrow\,\, \mathbb R$ be a
function, and set
\[
\label{cor57eq}
f^\blacktriangledown(x)=\sum_{\ell=0}^x \begin{pmatrix}x
\\  \ell\end{pmatrix}f(\ell)(-1)^{\ell-x}.
\]
Then, $f\in\mathcal H(K)$ if and only if the transform
$x\mapsto f^\blacktriangledown(x)$ is in $\ell^2(\mathbb Z_+)$.
\end{corollary}

\begin{proof}
In Theorem \ref{tm:inverse1} we proved that if $f\in\mathcal
H(K)$ then its transform $f^\blacktriangledown$ belongs to
$\ell^2(\mathbb Z_+)$. In fact the transform is an isometric
isomorphism of $\mathcal H(K)$ onto $\ell^2(\mathbb Z_+)$.\\

We turn to the converse: Let $f$ be a fixed function such that
$f^\blacktriangledown\in\ell^2(\mathbb Z_+)$. Recall (see Theorem
\ref{tm33}) that
\[
\label{100912}
e_n(x)=\sum_{j=0}^n (-1)^{j+n}\begin{pmatrix}n\\
j\end{pmatrix}K_j(x).
\]
Since the functions $e_n$ form an orthonormal basis of $\mathcal
H(K)$, we conclude that
\[
F:=\sum_{n=0}^\infty f^\blacktriangledown (n)e_n\in\mathcal H(K),
\]
and moreover,
\[
\langle F,e_n\rangle_{\mathcal
H(K)}=f^\blacktriangledown(n),\quad\forall n\in\mathbb Z_+.
\label{bvd-voltaire}
\]
Using \eqref{100912} again, together with the reproducing kernel
property for $F$ we conclude that
\[
\nonumber f(x)=F(x),\quad\forall x\in\mathbb Z_+,
\]
and so $f\in\mathcal H(K)$ as claimed.
\end{proof}

\begin{definition}
For $f\,\,:\,\,\mathbb Z_+\,\,\longrightarrow\,\, \mathbb R$ we define the following
Hurwitz transformation:
\[
\label{eq412}
f\mapsto H(f)(z)=\sum_{x=0}^\infty
\frac{f(x)}{x!}z^x=e^z\sum_{n=0}^\infty\frac{f^\blacktriangledown(n)}{n!}z^n,
\]
where $z\in\mathbb C$ is a generating function variable.
\label{defhurw}
\end{definition}

That the two expressions for $H(f)$ coincide can be checked as follows. For $n=0,1,\ldots$
we have the convolution expression:
\begin{equation}
\frac{f^\blacktriangledown(n)z^n}{n!}=\sum_{x=0}^n\frac{f(x)z^x}{x!}\cdot\frac{(-1)^{n-x}z^{n-x}}{(n-x)!}.
\label{revise}
\end{equation}
The result follows from the theorem on the coefficients of a product of power series.\\

%\frac{1}{x!}z^x=e^z\left(\sum_{n=x}^\infty\frac{
%\]

We have the following corollary to Corollary  \ref{cor57}.

\begin{corollary}
Let $f\in\mathcal H(K)$,and consider $f^\blacktriangledown\in\ell^2(\mathbb Z_+)$,
see \eqref{cor57eq}. Then the following hold for the respective Hurwitz-transforms:
\begin{equation}
H(f)(z)=e^zH(f^\blacktriangledown)(z).
\end{equation}
In particular, $H(f)$ is an entire function.
\label{cor??}
\end{corollary}

The following proposition will be used in Corollary \ref{cor??}.

\begin{proposition}
Assume $(f(x))_{x\in\mathbb Z_+}\in\ell^2(\mathbb Z_+)$. Then the
Hurwitz transform $H(f)$ is an entire function of
$z$.
\label{prop:republique}
\end{proposition}

\begin{proof}
Indeed, for any $z\in\mathbb C$, the sequence
\[
\nonumber
\left(\dfrac{z^x}{x!}\right)_{x\in\mathbb Z_+}\in\ell^2(\mathbb
Z_+).
\]
By the Cauchy-Schwarz inequality, we see that the series
\eqref{eq412} converges then absolutely for all complex numbers $z$.
\end{proof}

\begin{proof} By Corollary \ref{cor57}, $f\in\mathcal H(K)$ if and
only if there exists $(a_n)_{n\in\mathbb Z_+}\in\ell^2(\mathbb Z_+)$ such that
\[
f=\sum_{n=0}^\infty a_ne_n,
\]
with
\[
a_n=f^\blacktriangledown(n)
=\sum_{x=0}^n(-1)^{n+x}\begin{pmatrix}n\\ x\end{pmatrix}f(x).
\]
In a way similar to \eqref{revise} we have:
\[
\frac{f^\blacktriangledown(n)}{n!}=\sum_{x=0}^n\frac{f(x)}{x!}
\frac{(-1)^{n-x}}{(n-x)!}.
\]
By Proposition \ref{prop:republique}, the function $H(f^\blacktriangledown)$
is entire. Since
\[
\frac{f(n)}{n!}=\sum_{x=0}^n\frac{f^\blacktriangledown(x)}{x!}\frac{1}{(n-x)!},
\]
classical results on power series and convolution implies that
\[
H(f)(z)=(H(f^\blacktriangledown)(z))e^z,
\]
and $H(f)$ is in particular an entire function.
\end{proof}

\begin{proposition}
Let $f\,\,:\,\,\mathbb Z_+\,\,\longrightarrow\,\, \mathbb R$, and assume
that $f\in\mathcal H(K)$.
Then, for every $\epsilon>0$ there exists $M>0$  such that
\begin{equation}
\label{ineq123}
|f(n)|\le Mn!\epsilon^n,\quad n\in\mathbb Z_+.
\end{equation}
\end{proposition}
\begin{proof}
From the preceding proposition we know that $H(f)$ is an entire function.
Thus for every $r>0$ there exists an $M>0$ such
that
\[
\nonumber
r^n\frac{|f(n)|}{n!}\le M,\quad n\in\mathbb Z_+.
\]
The result follows with $r=1/\epsilon$.
\end{proof}

We note that \eqref{ineq123} is not sufficient to guarantee that
$f\in\mathcal H(K)$. Indeed, the series
\[
\nonumber
\sum_{n=0}^\infty \frac{f^\blacktriangledown(n)}{n!}z^n
\]
may define an entire function even when $(f^\blacktriangledown(n))_{n\in\mathbb Z_+}
\not\in\ell^2(\mathbb Z_+)$.
For example consider the function
$g_a$ given by:
\[
\label{eqga1}
g_a(y)=\frac{a^y}{y!},\quad y\in\mathbb Z_+.
\]
Then
\[
\nonumber
\begin{split}
g_a^\blacktriangledown(n)&=(-1)^n\sum_{x=0}^n\begin{pmatrix}n\\ x
\end{pmatrix}\frac{(-a)^x}{x!}\\
&=(-1)^n\sum_{x=0}^n(-a)^x\frac{n(n-1)\cdots (n-x+1)}{(x!)^2},
\end{split}
\]
and so $(g_a^\blacktriangledown(n))_{n\in\mathbb Z_+}\not\in\ell^2(\mathbb Z_+)$.
On the other hand,
\[
\nonumber
H(g_a^\blacktriangledown)(z)=e^{-z}(H(g_a)(z))=e^{-z}\sum_{n=0}^\infty
\frac{(az)^n}{(n!)^2}
\]
is entire. But $g_a\not\in\mathcal H(K)$ since
$(g_a^\blacktriangledown(n))_{n\in\mathbb Z_+}\not\in\ell^2(\mathbb Z_+)$.

\begin{corollary}
Let $f\,\,:\,\,\mathbb Z_+\,\,\longrightarrow\,\, \mathbb R$ be a
function, and assume that $f\in\mathcal H(K)$. Then there exists
$M<\infty$ such that
\[
\label{sydney}
|f(x)|^2\le M\begin{pmatrix}2x\\ x\end{pmatrix},\quad\forall
x\in\mathbb Z_+.
\]
(In other words, condition \eqref{sydney} is necessary for $f$ to
be in $\mathcal H(K)$).
\end{corollary}

\begin{proof}
Using the Cauchy-Schwarz inequality we get
\[
\begin{split}
\nonumber |f(x)|^2&\le \|f\|^2_{\mathcal H(K)}\cdot K(x,x)\\
&= \|f\|^2_{\mathcal H(K)}\cdot\begin{pmatrix}2x\\x\end{pmatrix}
\end{split}
\]
See \eqref{formula1} for the latter.
\end{proof}

\begin{remark}{\rm Let $a\in\mathbb R\setminus\left\{0\right\}$,
and let $g_a$ be given by \eqref{eqga1}, that is
\[
g_a(x)=\frac{a^x}{x!}, \quad x\in\mathbb Z_+,
\]
with as usual $0!=1$.
The condition \eqref{sydney} is satisfied, but
$g_a^\blacktriangledown$ is not in $\ell_2(\mathcal Z_+)$. We
conclude that condition \eqref{sydney} is not sufficient and
$g_a\not\in\mathcal H(K)$.}
\end{remark}

We now define an isometric isomorphism between the reproducing kernel Hilbert space
$\mathcal H(K)$ and the reproducing kernel Hilbert space $\mathcal H(K_2)$
of entire functions associated to the positive definite function
\begin{equation}
K_2(\zeta,z)=\sum_{n=0}^\infty\frac{\zeta^nz^{*n}}{(n!)^2},\quad \zeta,z\in\mathbb C.
\end{equation}
An orthonormal basis of the space $\mathcal H(K_2)$ is given by the functions
\[
\label{onbk2}
\zeta\mapsto \frac{\zeta^n}{n!},\quad n=0,1,\ldots
\]
For $f$ a real- or complex- valued function defined on $\mathbb Z_+$ we define the map
\[
(\mathscr H(f))(\zeta)=e^{-\zeta}(H(f))(\zeta).
\]
\begin{theorem}
\label{tm512}
The map $\mathscr H$ is unitary from $\mathcal H(K)$ onto $\mathcal H(K_2)$, i.e,
is isometric and onto.
\end{theorem}
\begin{proof}
Corollary \ref{cor57} implies that $\mathscr H$ maps $\mathcal H(K)$ into
$\mathcal H(K_2)$. Indeed $f$ belongs to $\mathcal H(K)$ if and only if
the transform $f^\blacktriangledown$ belongs to $\ell^2(\mathbb Z_+)$. But the
function
\[
(\mathscr H(f))(\zeta)=e^{-\zeta}\sum_{n=0}^\infty f(x)\frac{\zeta^x}{x!}=
\sum_{n=0}^\infty f^\blacktriangledown(n)\frac{\zeta^n}{n!}
\]
belongs to $\mathcal H(K_2)$ since  the functions \eqref{onbk2} form an orthonormal basis of $\mathcal H(K_2)$ and
$f^\blacktriangledown\in\ell^2(\mathbb Z_+)$.\\

To finish the proof, we show that the map $\mathscr H$ sends an orthonormal basis of $\mathcal H(K)$
onto an  orthonormal basis of $\mathcal H(K_2)$. We claim that
\[
(\mathscr H(e_n))(\zeta)=\frac{\zeta^n}{n!},\quad n=0,1,\ldots
\]
Indeed
\[
\begin{split}
(\mathscr H(e_n))(\zeta)&=e^{-\zeta}\sum_{x=n}^\infty e_n(x)\frac{\zeta^x}{x!}\\
&=e^{-\zeta}\frac{\zeta^n}{n!}\sum_{x=n}^\infty\frac{\zeta^{x-n}}{(x-n)!}\\
&=\frac{\zeta^n}{n!}
\end{split}
\]
concluding the proof.
\end{proof}

The following theorem establishes a co-product
for the binomial functions $e_n$ in \eqref{eqen}.
It follows from our analysis that the functions
$\left\{ e_n\, |\, n \in \mathbb N_0\right\}$  generate a hypergroup $HG$; in the sense of
Lasser et al, see \cite{MR2606478}. Here we refer to our formula \eqref{344}
below for the co-product in $HG$. And in \eqref{groningen2}
we give an explicit formula for the coefficients defining the co-product.

\begin{theorem}
\label{tot-le-matin}
\mbox{}\\
$(1)$ It holds that
\[
e_n^\blacktriangledown(x)=\delta_{n,x}.
\]
$(2)$ We have
\begin{equation}
\label{344}
e_n\cdot e_m=\sum_{k=m\vee n}^{m+n}(e_m\cdot
e_n)^\blacktriangledown(k)e_k,
\end{equation}
where
\[
\label{groningen2}
(e_m\cdot
e_n)^\blacktriangledown(k)=(-1)^{n+k}
\begin{pmatrix}k\\ n\end{pmatrix}
\left(\sum_{\ell=0}^{k-n}
\begin{pmatrix}k-n\\
\ell\end{pmatrix}\begin{pmatrix}\ell+n\\
m\end{pmatrix} \right).
\]
\label{groningen}
\end{theorem}

For example, we have
\[
e_1\cdot e_n=ne_n+(n+1)e_{n+1}.
\]
\begin{proof}[Proof of Theorem \ref{groningen}]
The first claim is a mere rewriting of the second equality in
\eqref{eqen1}. To prove \eqref{groningen2} we assume $m\le n$ (so that $m\vee n=n$) and write
\[
\nonumber
\begin{split}
(e_m\cdot
e_n)^\blacktriangledown(k)&=\sum_{y=n}^k(-1)^{y+k}
\begin{pmatrix}k\\ y\end{pmatrix}
\begin{pmatrix}y\\ n\end{pmatrix}
\begin{pmatrix}y\\ m\end{pmatrix}\\
&=\sum_{y=n}^k\begin{pmatrix}k\\ n\end{pmatrix}
\begin{pmatrix}k-n\\ y-n\end{pmatrix}
\begin{pmatrix}y\\ m\end{pmatrix}\\
&=
\begin{pmatrix}k\\ n\end{pmatrix}(-1)^{n+k}\left(\sum_{\ell=0}^{k-n}(-1)^\ell \begin{pmatrix}k-n\\ \ell\end{pmatrix}
\begin{pmatrix}\ell+n\\ m\end{pmatrix}\right)
\end{split}
\]
\end{proof}

\section{The second transform}
\label{sec7}
%
%The purpose in the present section is to give a geometric answer (Theorem \ref{tm88})
%to the characterization question (i) at the beginning of Section \ref{sec4}
%for $\mathcal H(K)$.

We show that the two classes of transforms in our study are two sides of a
harmonic analysis for the reproducing kernel Hilbert space $\mathcal H(K)$,
see Theorem \ref{tm11} and Theorem \ref{tm12} below. In the sense of computations,
a main distinction between the two is that the first class of transforms
involves only finite summations, while the second infinite; see e.g.,
Definition \ref{dn43}, and Proposition \ref{pn46}. In the terminology of the
infinite square matrices in Sections \ref{sec1} and \ref{lastsec},
the distinction reflects the difference in the algebra of infinite lower
triangular matrices, vs upper triangular ones.
In Theorem \ref{tm88}, we establish explicit inversion formulas.
Caution, while the range of the respective isometric transforms are Hilbert
spaces, they differ from one to the other.\\

We note that the function
\begin{equation}
\label{knew}
(-1)^{y+z}
K(y,z)=(-1)^{y}(-1)^zK(y,z)
\end{equation}
is still positive on $\mathbb Z_+$, and
recall that we denote by $\ell_0(\mathbb Z_+)$ the vector space of sequences from $\mathbb
Z_+$ into $\mathbb R$, with compact support. We endow $\ell_0(\mathbb Z_+)$ with the bilinear form
\[
\label{Chevaleret:ligne6}
\langle f,g\rangle_K=\sum_{x=0}^\infty\sum_{y=0}^\infty (-1)^{y+z}f(x)g(y)K(x,y).
\]
\begin{proposition}
The space $\ell_0(\mathbb Z_+)$ endowed with the bilinear form \eqref{Chevaleret:ligne6} is a pre-Hilbert space.
\end{proposition}

\begin{proof}
Bilinearity of $\langle \cdot,\cdot\rangle_K$ is clear from the definition, while the positivity property
\[
\langle f,f\rangle_K\ge 0,\quad\forall f\in\ell_0(\mathbb Z_+)
\]
follows from the fact that the function \eqref{knew} is positive definite on $\mathbb Z_+$.
We now show that the form
$\langle \cdot,\cdot\rangle_K$ is non-degenerate. Let $g\in\ell_0(\mathbb Z_+)$ be such that
\[
\langle f,g\rangle_K= 0,\quad\forall f\in\ell_0(\mathbb Z_+).
\]
In particular,
\[
\langle g,g\rangle_K=0.
\]
This cannot be since all the submatrices built from the kernel $K$ from different points
are strictly positive (see Theorem \ref{rksum2}).
\end{proof}

We denote by $\ell_K^2(\mathbb Z_+)$ the closure of $\ell_0(\mathbb Z_+)$ with respect to
$\langle \cdot,\cdot\rangle_K$.

\begin{proposition}
The map
\[
f\quad\mapsto\quad\sum_{y=0}^\infty K(x,y)(-1)^yf(y),
\]
first defined for $f\in\ell_0(\mathbb Z_+)$, extends to a unitary map from $\ell_K^2(\mathbb Z_+)$ onto
$\mathcal H(K)$.
\end{proposition}

\begin{proof}
Let $f\in \ell_0(\mathbb Z_+)$.We have (recall that all the sums are finite since $f$ has finite support)
\[
\begin{split}
\langle f,f\rangle_K&=\sum_{n,m=0}^\infty (-1)^n(-1)^mf(n)f(m)K(n,m)\\
&=\left\langle\sum_{n=0}^\infty K_n(-1)^nf(n),\sum_{m=0}^\infty K_m(-1)^mf(m)\right\rangle_{\mathcal H(K)}.
\end{split}
\]
The claim then follows from density and from the continuity of the inner product.
\end{proof}

\begin{definition}
Let $f\in\ell_0(\mathbb Z_+)$. We set
\[
\label{eq48}
f^\vartriangle(x)=\sum_{y=x}^\infty (-1)^{x+y}\begin{pmatrix}y\\ x\end{pmatrix}f(y).
\]
\label{dn43}
\end{definition}
We note that \eqref{eq48} can be rewritten as
\[
f^\vartriangle=M(-1)f,
\]
where $M(\lambda)$ is defined by \eqref{eqm}.\\

Explicitly, the series is written as
\[
\begin{split}
f^\vartriangle(x)&=(-1)^x\left(f(x)-(x+1) f(x+1)+\frac{(x+2)(x+1)}{2}f(x+2)-\right.\\
&\left.\hspace{20mm}-\frac{(x+3)(x+2)(x+1)}{6}f(x+3)+\cdots\right)
\end{split}
\]
Note that $f^\vartriangle\in\ell_0(\mathbb Z_+)$, and more precisely, if
$f(x)=0$ for $x\ge N$ then $f^\vartriangle(x)=0$ for $x\ge N$.

\begin{example}
Let $g_a$ be defined by \eqref{eqga1}: $g_a(y)=\frac{a^y}{y!}$, where $a\in\mathbb R$.
Then,
\[
\label{eqga}
g_a^\vartriangle=e^{-a}g_a.
\]
\end{example}
Indeed,

\[
\begin{split}
\nonumber
g_a^\vartriangle(x)&=\sum_{y=x}^\infty(-1)^{x+y}\begin{pmatrix}y\\
x\end{pmatrix}
\frac{a^y}{y!}\\
&=\sum_{y=x}^\infty(-1)^{x+y}\frac{a^y}{x!(y-x)!}\\
&=\frac{1}{x!}\sum_{y=x}^\infty(-1)^{x+y}\frac{a^y}{(y-x)!}\\
%\intertext{\rm
&=\frac{a^x}{x!}\sum_{u=0}^\infty(-1)^{u}\frac{a^u}{u!}\\
&=e^{-a}\frac{a^x}{x!}.
\end{split}
\]
where we have used the change of index $y-x=u$ to go from the third
to the fourth line.\\

We note the following: though $g_a$ does not belong to $\mathcal H(K)$ unless $a=0$, the $\Delta$ transform
of $g_a$ exists for all real $a$.\\

In preparation of Theorem \ref{tm12} we introduce the notation
$s(f)$ as:
\[
\label{eq:sf}
f\quad \mapsto \quad s(f):=\sum_{x\in\mathbb Z_+} K_xf(x)
\]

\begin{theorem}
\label{tm12}
The map $f\mapsto s(f^\vartriangle)$ extends to a unitary map
from $\mathcal H(K)$ onto $\ell^2(\mathbb Z_+)$.
\end{theorem}

\begin{proof}
Let $f\in\ell_0(\mathbb Z_+)$. We have (recall that all the sums are
finite since $f$ has finite support):
\[
\nonumber
\begin{split}
\langle s(f^\vartriangle), s(f^\vartriangle)\rangle_{\mathcal H(K)}&=\sum_{x_1=0}^\infty\sum_{x_2=0}^\infty
s(f^\vartriangle)(x_1)s(f^\vartriangle)(x_2)K(x_1,x_2)\\
&=\sum_{x_1=0}^\infty\sum_{x_2=0}^\infty \left(\sum_{y=x_1}^\infty
(-1)^{x_1+y}\begin{pmatrix}y\\ x_1\end{pmatrix}f(y)\right)
\times\\
&\hspace{20mm}\times K(x_1,x_2)
\left(\sum_{z=x_2}^\infty (-1)^{x_2+z}\begin{pmatrix}z\\ x_2\end{pmatrix}f(z)\right)\\
&=\sum_{x_1=0}^\infty\sum_{x_2=0}^\infty\sum_{y=x_1}^\infty
\sum_{z=x_2}^\infty\sum_{j=0}^{x_1\wedge x_2}(-1)^{x_1+x_2+y+z}
f(y)f(z)\times\\
&\hspace{39mm}\times
\begin{pmatrix}y\\ x_1\end{pmatrix}\begin{pmatrix}z\\ x_2\end{pmatrix}
\begin{pmatrix}x_1\\ j\end{pmatrix}\begin{pmatrix}x_2\\ j\end{pmatrix}\\
&=\sum_{y=0}^\infty\sum_{z=0}^\infty
f(y)f(z)\left(\sum_{j=0}^\infty\left(\sum_{x_1=j}^y(-1)^{x_1+y}
\begin{pmatrix}y\\ x_1\end{pmatrix}\begin{pmatrix}x_1\\
j\end{pmatrix}\right)\times\right. \\
&\left.\hspace{35mm}\times
 \left(\sum_{x_2=j}^z (-1)^{x_2+z}
\begin{pmatrix}z\\ x_2\end{pmatrix}\begin{pmatrix}x_2\\ j\end{pmatrix}\right)\right)\\
\intertext{\mbox{\rm and using Lemma \ref{le:eq}}}
&=\sum_{j=0}^\infty
\sum_{y=0}^\infty\sum_{z=0}^\infty f(y)f(z)\delta_{j,y}\delta_{j,z}\\
&=\sum_{j=0}^\infty f(j)^2.
\end{split}
\]
\end{proof}

We set $\mathscr A(f)=s(f^\vartriangle)$. It turns out (see the next proposition),
that although
$s(f)$ is difficult to compute for general value of $f$, the operator
$\mathcal A$ and its adjoint are easier
to handle.

\begin{proposition} The adjoint $\mathscr A^*$ of the operator $\mathscr A$ is given by:
\begin{equation}
\left(\mathscr A^*(\varphi)\right)(x)=\sum_{y=x}^\infty(-1)^{x+y}
\begin{pmatrix}y\\x\end{pmatrix}\varphi(y),\quad \varphi\in\mathcal H(K),
\end{equation}
and with values in $\ell^2(\mathbb Z_+)$.
\label{pn46}
\end{proposition}
\begin{proof} We check the result for $\varphi_{K_{y_0}}$. The general case
follows then by continuity. Let $f\in\ell^2(\mathbb Z_+)$. We have:
\[
\begin{split}
\langle\mathscr Af,\varphi\rangle_{\mathcal H(K)}&=(\mathscr Af)(y_0)\\
&=\sum_{t=0}^\infty f^\vartriangle(t)K(y_0,t)\\
&=\sum_{t=0}^\infty\sum_{s=t}^\infty (-1)^{t+s}f(s)\begin{pmatrix}s\\t\end{pmatrix}
K(y_0,t)\\
&=\sum_{s=0}^\infty f(s)(-1)^{s}\sum_{t=0}^s \begin{pmatrix}s\\t\end{pmatrix}
K(y_0,t),
\end{split}
\]
and hence the result.
\end{proof}

\begin{definition}
Let $g\in\ell_0(\mathbb Z_+)$. We set
\[
g^{\triangledown}(n)=\sum_{x=n}^\infty
\begin{pmatrix}x\\n\end{pmatrix}g(x).
\]
\end{definition}

\begin{theorem}
\label{tm88}
The map $g\mapsto g^\triangledown$ extends to a unitary map from
$\ell_K^2(\mathbb Z_+)$
onto $\ell^2(\mathbb Z_+)$, and its adjoint is the map
$f\mapsto f^\vartriangle$. In particular:
\[
\begin{split}
(g^\triangledown)^\vartriangle&=g,\quad\forall g\in\ell_K^2(\mathbb Z_+) \\
%\intertext{{and}}
\left(f^\vartriangle\right)^\triangledown&=f,\quad\forall
f\in\ell^2(\mathbb Z_+).
\end{split}
\]
\end{theorem}

\begin{proof} We divide the proof in two steps.\\

STEP 1: {\sl
We first prove that the two transforms are inverse of each other. Then we show that
they are adjoint of each other.}\\

Let $g\in\ell_0(\mathbb Z_+)$. We have:
\begin{eqnarray*}
(g^\triangledown)^\vartriangle(y)&=&\sum_{n=0}^\infty
(-1)^{n+y}\begin{pmatrix}n\\ y\end{pmatrix}(g^\vartriangle(n))\\
&=&\sum_{n=y}^\infty(-1)^{n+y}\begin{pmatrix}n\\ y\end{pmatrix}\left(\sum_{x=n}^\infty\begin{pmatrix}x\\ n\end{pmatrix}
g(x)\right)\\
&=&\sum_{x=y}^\infty\left(\sum_{n=y}^x(-1)^{n+y}\begin{pmatrix}n\\ y\end{pmatrix}\begin{pmatrix}x\\ n\end{pmatrix}\right)\\
&=&g(y)
\end{eqnarray*}
since (see Lemma \ref{le:eq})
\[
\sum_{n=y}^x(-1)^{n+y}\begin{pmatrix}n\\ y\end{pmatrix}\begin{pmatrix}x\\ n\end{pmatrix}=\delta_{x,y}.
\]
Similarly, we have for $f\in\ell^2(\mathbb Z_+)$ and $n\ge 0$:

\begin{eqnarray*}
\left((f^\vartriangle)^\triangledown\right)(n)&=&\sum_{x=n}^\infty \begin{pmatrix}x\\ n\end{pmatrix} f^\vartriangle(x)\\
&=&\sum_{x=n}^\infty\begin{pmatrix}x\\ n\end{pmatrix}\left(\sum_{m=x}^\infty \begin{pmatrix}m\\ x\end{pmatrix}f(m)\right)\\
&=&\sum_{m=n}^\infty f(n)\left(\sum_{x=n}^m(-1)^{x+m}
\begin{pmatrix}m\\ x\end{pmatrix}\begin{pmatrix}x\\ n\end{pmatrix}\right)\\
&=&f(m)
\end{eqnarray*}
since (see Lemma \ref{le:eq})
\[
\sum_{x=n}^m(-1)^{x+m}
\begin{pmatrix}m\\ x\end{pmatrix}\begin{pmatrix}x\\ n\end{pmatrix}=\delta_{m,n}.
\]

STEP 2: {\sl We  now prove that the two transforms are adjoint of each other.}\\

Let $f\in\ell^2(\mathbb Z_+)$ and
$g\in\ell_K^2(\mathbb Z_+)$. We have:
\[
\begin{split}
\langle g,
f^\vartriangle\rangle_{\ell_K^2(\mathbb Z_+)}&=\sum_{x,y=0}^\infty g(x)f^\vartriangle
(y)K(x,y)\\
&=\sum_{x,y=0}^\infty g(x)\left(\sum_{t=y}^\infty (-1)^{t+y}f(t)\begin{pmatrix}t\\ y
\end{pmatrix}\right)K(x,y)\\
&=\sum_{t=0}^\infty f(t)\left(\sum_{x=0}^\infty g(x)\sum_{y=0}^t
(-1)^{t+y}\begin{pmatrix}t\\ y
\end{pmatrix}\right)K(x,y)\\
&=\sum_{t=0}^\infty f(t)\left(\sum_{x=t}^\infty g(x)\begin{pmatrix}x\\ t
\end{pmatrix}\right)
\intertext{\rm where we use \eqref{eq:inverse11}}
&=\langle g^\triangledown f\rangle_{\ell^2(\mathbb Z_+)}
\end{split}
\]

To conclude we now prove that the $\triangledown$ transform is an isometry.
We first prove this on the space $\mathcal H_0$ of
functions from $\mathbb Z_+$ into $\mathbb R$ with finite support. For $g\in \ell_0$
we have:
\[
\begin{split}
\|g^\vartriangle\|^2_{\mathcal H(K)}&=\sum_{n=0}^\infty\left(\sum_{x=n}^\infty
\begin{pmatrix}x\\ n\end{pmatrix}g(x)\right)^2\\
&=\sum_{x,y=0}^\infty g(x)g(y)\left(\sum_{n=0}^{x\wedge y}
\begin{pmatrix}x\\ n
\end{pmatrix}\begin{pmatrix}y\\ n
\end{pmatrix}\right)\\
&=\sum_{x,y=0}^\infty g(x)g(y)K(x,y)\\
&=\|g\|^2_{\ell^2(K)}.
\end{split}
\]
\end{proof}

\section{Motivation from discrete analytic functions}
\label{sec8}
Here we offer a brief outline of the significance of our binomial
functions in the study of discrete analytic functions. While our
binomial functions are functions of a single discrete variable
(here  $\mathbb Z_+$), we show in \cite{2012arXiv1208.3699A} that
the binomial functions extend uniquely to discrete analytic
functions on the $2D$ lattice, so formulas obtained in the earlier
part of our paper have direct implications for Hilbert spaces of
discrete analytic functions.\\

Recall first that a function $f:\ZZ^2\longrightarrow\CC$ is said to
be discrete analytic if
\[
\label{DCR} \forall (x,y)\in \ZZ^2,\qquad
\dfrac{f(x+1,y+1)-f(x,y)}{1+i}=\dfrac{f(x+1,y)-f(x,y+1)}{1-i}.
\]
See \cite{MR0013411,MR0078441}.\\

Given a function
$f_0\,\,:\,\,\ZZ\,\,\longrightarrow\,\,\mathbb C$ there are
infinitely discrete analytic functions $f$ on $\ZZ^2$ such that
$f(x,0)=f_0(x)$. However,when $f_0$ is a polynomial, only one of these discrete
analytic extensions will be a polynomial in $x,y$. See \cite{MR0078441} (we gave a new proof of this
fact in \cite{2012arXiv1208.3699A}). Thus,
there exists a unique discrete analytic polynomial
$\zeta_n(x,y)$ determined by $\zeta_n(x,0)\equiv x^{[n]}$, where $x^{[n]}$ has
been defined in \eqref{xnstr}.

\section{The quantum case}
\setcounter{equation}{0}
\label{sec9}
In a number of questions from mathematical physics
(see \cite{MR2886683,MR2821778,MR1284951} and the papers cited there),
it is of interest to consider the questions from sections \ref{sec3} through
\ref{sec5} for a deformation family of the binomial functions
\eqref{eqen}. Traditionally the complex deformation parameter is denoted $q$ ,
and the $q$-versions of the binomial formulas were first studied in
\cite{MR0027288}. The purpose of this section is to prove an analogue of the orthonormal basis
theorems but adapted to the $q$-binomials.\\

The $q$-deformed binomial coefficients were defined by Carlitz;
see \cite{MR0095982,MR0027288}. We fix $q\in\mathbb
R\setminus\left\{1\right\}$, and set
\[
\begin{split}
[x]&:=\frac{q^x-1}{q-1},\\
[x]_n&:=[x][x-1]\cdots [x-n+1],\\
[n]!&:=[n]_n,\quad {\rm and}\quad [0]!=[0]=1,\\
\left[\begin{matrix} x\\
n\end{matrix}\right]&:=\frac{[x]_n}{[n]!},
\end{split}
\]
where all the expressions depend on the value of $q$. To lighten
the notation, we do not stress the $q$-dependence. We introduce
\[
\label{eqenq}
e_n(x)=\begin{cases}\,\,0\,,
\quad\hspace{2.cm}{\rm if}\quad x<n,\\
&\\
\left[\begin{matrix} x\\
n\end{matrix}\right]=\frac{[x]_n}{[n]!},\hspace{1.15cm}{\rm if}
\quad x\ge n,
\end{cases}
\]
and
\[
\label{rksum1} K(x,y)=\sum_{n=0}^\infty e_n(x)e_n(y), \quad
x,y\in\mathbb Z_+.
\]
The sum \eqref{rksum1} is always finite and well defined, and
defines a positive definite function on $\mathbb Z_+$. It can also
be rewritten as:
\[
\label{rksum21}
K(x,y)=\sum_{n=0}^{x\wedge y}
\left[\begin{matrix} x\\
n\end{matrix}\right]\left[\begin{matrix} y\\
n\end{matrix}\right].
\]

The functions $e_0,e_1,\ldots$ are easily seen to be linearly independent, and
so we have:

\begin{theorem}
The functions $e_0, e_1,\ldots$ form an orthonormal basis of
$\mathcal H(K)$.
\end{theorem}

\section{A one parameter group of transforms}
\setcounter{equation}{0}
\label{lastsec}
Let $\lambda\in(0,\infty)$. The function
\[
\label{klambdaker}
K_\lambda(x,y)=\sum_{n=0}^{x\wedge y}\lambda^{x+y-2n}\begin{pmatrix}x\\ n \end{pmatrix}
\begin{pmatrix}y\\ n \end{pmatrix}
\]
is positive definite in $\mathbb Z_+$, and we denote by
$\mathcal H(K_\lambda)$ the associated reproducing kernel Hilbert space.
The functions
\[
e^{(\lambda)}_n(x)=\lambda^{x-n}e_n(x),\quad, n=0,1,\ldots
\]
form an orthonormal basis of the associated reproducing kernel Hilbert space
$\mathcal H(K_\lambda)$
Note that the case $\lambda=1$ corresponds to formulas \eqref{rksum2} and
\eqref{eqen}. Let $L(\lambda)$ denote the Zadeh-transform of the lower triangular
infinite matrix $L$ given by \eqref{eq:l}, that is the lower triangular infinite
matrix $M(\lambda)$ as follows: For a point $(x,y)\in\mathbb Z_+\times \mathbb Z_+$,
let $x$ denote row-index and $y$ denote column-index. The matrix $L(\lambda)$
is defined by:
\[
\label{eq1}
\left(L(\lambda)\right)_{(x,y)}
=\begin{cases}\lambda^{x-y}\begin{pmatrix}x\\ y\end{pmatrix},\quad
{\rm if}\quad 0\le y\le x\\
\,\, 0,\quad\hspace{1.4cm}{\rm otherwise.}\end{cases}
\]
Thus, $L(\lambda)$ is given by
\[
\label{petites-ecuries}
\left(
\begin{array}{ccccccccccc}
1&0&0&0&0&0&0&\cdots&\quad&\cdot&\cdots\\
\lambda&1&0&0&0&0&0&\cdots&\quad&\cdot&\cdots\\
\lambda^2&2\lambda&1&0&0&0&0&\cdots&\quad&\cdot&\cdots\\
\lambda^3&3\lambda^2&3\lambda&1&0&0&0&\cdots&\quad&\cdot&\cdots\\
\lambda^4&4\lambda^3&6\lambda^2&4\lambda&1&0&0&\cdots&\quad&\cdot&\cdots\\
\vdots&\vdots& & & & & & & & & \\
\lambda^n&n\lambda^{n-1}&
\lambda^{n-2}\begin{pmatrix}n\\2\end{pmatrix}&\cdots&&n\lambda&1&0
&0&\cdots&\cdots \\
\lambda^{n+1}&(n+1)\lambda^n&
\lambda^{n-1}\begin{pmatrix}n+1\\2\end{pmatrix}&\cdots&&
&(n+1)\lambda&1&0
&\cdots&\cdots \\
\vdots&\vdots&\vdots&&&\vdots&\vdots&\vdots&1&&\\
\vdots&\vdots&\vdots&&&\vdots&\vdots&\vdots&\vdots&&
\end{array}
\right).
\]
Then the matrix product $L(\lambda)L(\mu)$ makes sense for all
$\lambda,\mu\in\mathbb C$, with
\[
\label{yomadin1}
(L(\lambda)L(\mu))_{(x,y)}=\sum_{k=y}^x L(\lambda)_{(x,k)}L(\mu)_{(k,y)},\quad
\lambda,\mu\in\mathbb C.
\]
\begin{lemma}
With the notation above it holds that
\[
\label{grouplaw}
L(\lambda+\mu)=L(\lambda)L(\mu),\quad \lambda,\mu\in\mathbb C,
\]
i.e. the matrices in \eqref{eq1} define a one-parameter semi-group of infinite
matrices.
\end{lemma}

\begin{proof}
We give the details, but note that the computation of \eqref{yomadin1},
is essential the same as that given in the proof of Lemma \ref{le:eq} above. In other
words, when the product in \eqref{yomadin1} is computed, the group law \eqref{grouplaw}
follows. More precisely,
\[
\label{eq581}
\begin{split}
L(\lambda)L(\mu)&=D(\lambda)L(1)D(\lambda)^{-1}D(\mu)L(1)D(\mu)^{-1}\\
                &=D(\lambda)L(1)D\left(\frac{\mu}{\lambda}\right)L(1)D(\mu)^{-1}
\end{split}
\]
But for $y\le x$ we have:
\[
\nonumber
\begin{split}
\left(L(1)D\left(\frac{\mu}{\lambda}\right)L(1)\right)_{(x,y)}&=\sum_{k=y}^x
\begin{pmatrix}x\\ k\end{pmatrix}\left(\frac{\mu}{\lambda}\right)^k\begin{pmatrix}
k\\ y\end{pmatrix}\\
\intertext{\mbox{\rm and using $\begin{pmatrix}x\\ k\end{pmatrix}
\begin{pmatrix}k\\ y\end{pmatrix}=\begin{pmatrix}x\\ y\end{pmatrix}
\begin{pmatrix}x-y\\ k-y\end{pmatrix}$, this equality is equal to}}
&=\begin{pmatrix}x\\ y\end{pmatrix}\left(\sum_{k=y}^x
\left(\frac{\mu}{\lambda}\right)^k\begin{pmatrix}
x-y\\ k-y\end{pmatrix}\right)\\
&=\left(1+\frac{\mu}{\lambda}\right)^{x-y}
\left(\frac{\mu}{\lambda}\right)^y
\begin{pmatrix}x\\ y\end{pmatrix}.
\end{split}
\]
Thus
\[
L(1)D\left(\frac{\mu}{\lambda}\right)L(1)=D\left(1+\frac{\mu}{\lambda}\right)L(1)
D\left(\frac{\mu}{\mu+\lambda}\right),
\]
and this, together with \eqref{eq581} leads to the result.
\end{proof}

\begin{corollary}
The one-parameter group $\lambda\mapsto L(\lambda)$ of infinite matrices
in \eqref{petites-ecuries} has the form
\[
\label{manege}
L(\lambda)=\exp (\lambda A),
\]
where $A$ is the following sub-diagonal banded matrix
\[
A_{(x,y)}=x\delta_0(x-y-1),
\]
see \eqref{eqA}.
\end{corollary}

\[
\label{eqA}
\left(
\begin{array}{ccccccccccc}
0&0&0&0&0&0&0&\cdots&\quad&\cdot&\cdots\\
1&0&0&0&0&0&0&\cdots&\quad&\cdot&\cdots\\
0&2&0&0&0&0&0&\cdots&\quad&\cdot&\cdots\\
0&0&3&0&0&0&0&\cdots&\quad&\cdot&\cdots\\

\vdots&\vdots& & & & & & & & & \\
0&0&
\cdots&\cdots&&n-1&0&0
&0&\cdots&\cdots \\
0&0&
0&\cdots&&
&n&0&0
&\cdots&\cdots \\
0&0&
\cdots&\cdots&&0&0&n+1
&0&\cdots&\cdots \\
%0&0&
%0&\cdots&&
%&0&n+1&0
%&\cdots&\cdots \\
%\vdots&\vdots&\vdots&&&\vdots&\vdots&\vdots&&&\\
\vdots&\vdots&\vdots&&&\vdots&\vdots&\vdots&\vdots&&
\end{array}
\right).
\]

We denote
\[
\label{eqm}
M(\lambda)=L(\lambda)^*.
\]

We define
\[
(T_{-\lambda}(e_n^{(\lambda)}))(x)=\delta_{n,x}.
\]

The proof of the following theorem is easy and will be omitted.

\begin{theorem}
The map
\[
T_{-\lambda}\,\,:\,\, f\mapsto L(\lambda)^*f
\]
is an isometry from $\mathcal H(K_\lambda)$ into $\ell^2(\mathbb Z_+)$.
\end{theorem}

As a corollary we have:

\[
\label{lycee-voltaire}
\left(L(\lambda)L(\lambda)^*\right)_{(x,y)}=K_\lambda(x,y)=\sum_{n=0}^{x\wedge y}
e^{(\lambda)}_n(x)e^{(\lambda)}_n(y)^*,
\]
where $e^{(\lambda)}_n(x)=\lambda^{x-n}e_n(x)$.\\

Before stating another corollary to this theorem we mention the following lemma, whose proof is left to the reader.

\begin{lemma}
Let $(X_x(\cdot))$  be a Gaussian process with index $x$, with mean zero and probability space $(\Omega,\mathcal F, P)$.
Let $f$ be a non vanishing function on the index set. Then, the Gaussian process $(Y_x(\cdot))=(f(x)X_x(\cdot))$ has the same
probability space but with covariance $\mathbb E(Y_xY_y)=f(x)\mathbb E(X_xX_y)f(y)$.
\end{lemma}

\begin{corollary}
Let $(\Omega, \mathcal F,P^{\lambda})$  be probability spaces  and associated stochastic processes $X^{(\lambda)}$ such that
\[
\mathbb E_\lambda (X_x^{(\lambda)}X_y^{(\lambda)})=K_{\lambda}(x,y)
\]
Then the probability measures $P^{\lambda}$ are mutually singular, as we vary $\lambda\in(0,\infty)$.
\end{corollary}

\begin{proof} By the above lemma, the process $Y_x^{(\lambda)}=\frac{1}{\lambda^x}X_x^{(\lambda)}$ has the same probability
space but covariance function
\[
\mathbb E_\lambda (Y_x^{(\lambda)}Y_y^{(\lambda)})=\frac{1}{\lambda^x\lambda^y}K_{\lambda}(x,y).
\]
We write for $\lambda_1$ and $\lambda_2\in(0,\infty)$
\[
\begin{split}
\frac{1}{\lambda_1^x\lambda_1^y}K_{\lambda_1}(x,y)&=\sum_{n=0}^\infty \lambda_1^{-2n}e_n(x)e_n(y)\\
\frac{1}{\lambda_2^x\lambda_2^y}K_{\lambda_2}(x,y)&=\sum_{n=0}^\infty \lambda_2^{-2n}e_n(x)e_n(y)\\
&=\sum_{n=0}^\infty \mu_n\lambda_1^{-2n}e_n(x)e_n(y),
\end{split}
\]
where we have set $\mu_n=\frac{\lambda_2^{2n}}{\lambda_1^{2n}}$.
By \cite[Theorems 4.3 and  4.4, p.27]{MR0277027}, we see that $P^{\lambda_1}$ and $P^{\lambda_2}$ are equivalent if and only if
\[
\sum_{n=0}^\infty(1-\mu_n)^2<\infty,
\]
which holds if and only if $\lambda_1=\lambda_2$.
\end{proof}

{\bf Acknowledgments:} It is a pleasure to thank the referee for his/her thorough reading of the paper and very useful comments.

\bibliographystyle{amsplain}
%\bibliography{/users/faculty/math/dany/Travaux_courants/bib/all}
%bibliography{all}
\def\cprime{$'$} \def\lfhook#1{\setbox0=\hbox{#1}{\ooalign{\hidewidth
  \lower1.5ex\hbox{'}\hidewidth\crcr\unhbox0}}} \def\cprime{$'$}
  \def\cprime{$'$} \def\cprime{$'$} \def\cprime{$'$} \def\cprime{$'$}

\end{document}